\documentclass[11pt]{article}
\usepackage{amsmath,amssymb,amsthm,amscd}

\def \N {\mathbb{N}}
\def \Z {\mathbb{Z}}
\def \Q {\mathbb{Q}}

\def \c {\textbf{c}}
\def \J {{\cal J}}
\newtheorem{theorem}{Theorem}[section]
\newtheorem{proposition}[theorem]{Proposition}
\newtheorem{lemma}[theorem]{Lemma}

\def\GL{{\rm GL}}
\def\Aut{{\rm Aut}}
\newcommand{\nd}{\hbox{$\not\vert\,$}}

\begin{document}

\begin{center}
{\large \sc Finitely generated soluble groups and their subgroups}

\vspace{0.5cm}
{ Tara Brough and Derek F. Holt}

\end{center}

We prove that every finitely generated soluble group that is not virtually\\
abelian has a subgroup of one of a small number of types.\\

\section{Introduction}

The aim of this paper is to describe a small number of types of groups, which
have the property that every finitely generated soluble group that is not
virtually abelian has a subgroup of one of these types.

In Section~\ref{twogenmetasec}, we shall define a class of 2-generator
abelian-by-cyclic groups, each of which can be embedded into a semidirect
product $\Q^s\rtimes \Z$ for some ${s > 0}$. We shall call these
{\em Gc-groups} and prove in Theorem~\ref{metabelian} that every finitely
generated metabelian group that is not virtually abelian has a subgroup that
is either a \emph{proper} (that is, non virtually abelian) Gc-group or is isomorphic 
to $\Z \wr \Z$ or to $C_p \wr \Z$ for
some prime $p$, where $C_p$ denotes a cyclic group of order $p$.
Our second main result, Theorem~\ref{tfsol},
is that every finitely generated torsion-free
soluble group that is not virtually abelian has a subgroup that is either
a proper Gc-group or is isomorphic to $\Z^\infty$, which in this paper will denote
a free abelian group of countably infinite rank. Our main result, Theorem
\ref{main}, is a little less satisfactory. We show that
every finitely generated soluble group that is not virtually abelian
has a subgroup $H$ that is either a proper Gc-group or is isomorphic to
$\Z^\infty$, or is finitely generated with an infinite normal torsion
subgroup. We would prefer to be able to further restrict the structure of $H$ in
the third case. 

These results were motivated by the work of the first author on the class of
finitely generated groups whose word problem 
is an intersection of finitely many context-free languages - known as \emph{poly-${\cal CF}$} groups.  
Finitely generated virtually abelian groups are known to all be poly-${\cal CF}$, and it is conjectured 
that these are the only soluble poly-${\cal CF}$ groups.  Since the class of poly-${\cal CF}$ groups
is closed under taking finitely generated subgroups, it would suffice to find a list of finitely generated 
non-poly-${\cal CF}$ groups such that every finitely generated soluble group either has a subgroup 
isomorphic to one on the list or is virtually abelian.
The first author has shown \cite{TB} that $C_p\wr \Z$ for any prime $p$, proper Gc-groups and all groups containing
$\Z^\infty$ are not poly-${\cal CF}$.  Thus Theorems~\ref{metabelian} and~\ref{tfsol} lead to a proof of the
conjecture in the metabelian and torsion-free cases respectively.  It remains to deal with the groups in the
last case of Theorem~\ref{main}.  One approach might be to show that a poly-${\cal CF}$ group cannot have an
infinite torsion subgroup.

In general not much can be said about the structure of finitely generated
soluble groups and their subgroups. For example, B.H. Neumann and
H. Neumann proved in 1959~\cite[Section 4.1.1]{Rob04} that every 
countable soluble group of derived length $d$ may be embedded in a $2$-generator
soluble group of derived length at most $d+2$, and P. Hall proved in
1954~\cite[Section 4.1.3]{Rob04} that, for any non-trivial countable abelian
group $A$, there are $2^{\aleph_0}$ non-isomorphic $2$-generator
centre-by-metabelian groups $G$ such that $A\cong Z(G) = G''$.

We conclude this introductory section by stating some results
that we shall need from the literature.

\begin{proposition}\label{smithform}{\rm \cite[Theorem 9.64]{Rotman}}
Let $\Sigma = \mathrm{diag}(\sigma_1,\ldots,\sigma_k)$ be the
non-singular diagonal block 
in the Smith normal form of a matrix $M$ with entries in a euclidean ring $R$.  
Let $\gamma_0(M) = 1$, and for $i>0$ let $\gamma_i(M)$
be the $\gcd$ of all $i\times i$ minors of $M$.
Then $\sigma_i = \gamma_i(M)/\gamma_{i-1}(M)$ for all $1\leq i\leq k$.
\end{proposition} 

\begin{proposition}\label{G/N}
{\rm (Philip Hall: for a proof, see~\cite[14.1.3]{Rob82}.)}
Let $G$ be a finitely generated group, let $N\lhd G$ and suppose that $G/N$ is finitely presented.  Then $N$ is the normal closure in $G$ of a finite subset of $G$.
\end{proposition}

\begin{proposition}\label{schur}
{\rm (Schur: for a proof, see \cite[10.1.4]{Rob82}.)}
Let $G$ be a group with $|G : Z(G)|$ finite.  Then $G'$ is finite.
\end{proposition} 

\begin{proposition}\label{GLbound}
{\rm (Minkowski:  for a proof, see~\cite{RocTan}.)}
There exists a function $L: \N\rightarrow \N$ such that the order of any finite
subgroup of $\GL(n,\Z)$ divides $L(n)$.
\end{proposition}

\section{Two-generator metabelian groups}\label{twogenmetasec}

The proof our main result on metabelian groups will rely on some facts about a certain type of two-generator metabelian group.  
These are groups $H$ satisfying the following hypothesis.

\textbf{Hypothesis $(*)$}:  $H = \langle a,b\rangle$, where $|a|$ is infinite,
the subgroup $B = \left\langle b^{a^i} \mid i\in \Z\right\rangle$ is abelian, and $H$ is not virtually abelian.

We observe that $(*)$ implies that $H = B\rtimes \langle a\rangle$.

Throughout our discussion of metabelian groups, $H$ and $B$ will have these definitions, 
and for convenience we will use the notation $b_i$ for $b^{a^i}$.  
We will also sometimes use additive notation within $B$.

We will first consider two types of groups satisfying Hypothesis $(*)$: 
(i) $H$ is torsion-free;
(ii) $b$ is a torsion element;
before going on to the general case.

\subsection{Case 1:  $H$ is torsion-free}

Suppose that $H = \left\langle a,b\right\rangle$ is torsion-free and satisfies
Hypothesis $(*)$.  If $H\not\cong \Z \wr \Z$, then $B$ is not free abelian on
$\{\,b_i \mid i\in \Z\,\}$, so there exist integers $r, s, c_i$ with $r \le s$ and
$c_r,c_s \ne 0$ such that $\sum_{i=r}^s c_i b_i = 0.$ Since we are assuming that $H$ is torsion-free and not virtually abelian, we must have $r<s$, and
we may assume that $\gcd(c_0,\ldots,c_s) = 1$.
Furthermore, by conjugating by $a^{-r}$, we may take $r = 0$.

We will eventually show that this relation $\sum_{i=0}^s c_ib_i = 0$ 
(or in multiplicative notation $b_0^{c_0} b_1^{c_1}\cdots b_s^{c_s} = 1$), 
together with the fact that $B$ is abelian, fully determines $H$.  
To do this, we first consider a certain type of group presentation,
and then show that $H$ has a presentation of this type.

\subsubsection{The groups $G(\c)$}

For $\c = (c_0,\ldots,c_s)\in \Z^{s+1}$ with $s \ge 1$, $c_0,c_s\neq 0$ and
$\gcd(c_0,\ldots,c_s) = 1$,
let $G(\c)$ be the group defined by the presentation
$\langle \alpha, \beta \mid {\cal R}_{\c}\rangle$, where
\[{\cal R}_{\c} = \left\{[\beta, \beta^{\alpha^i}] \; (i\in \Z), \; \beta^{c_0}(\beta^{\alpha})^{c_1}\cdots(\beta^{\alpha^s})^{c_s}\right\}.\]
We shall call such groups {\em Gc-groups}, and when we refer to the
Gc-group $G(\c)=\langle x,y\rangle$, we will assume that
$\c \in \Z^{s+1}$ satisfies the above conditions, and that
$x$ replaces $\alpha$ and $y$ replaces $\beta$ in the above definition of
$G(\c)$.

Note that the Gc-groups include some polycyclic groups (when
$|c_0|=|c_s|=1$, for example), some of which are even virtually abelian. But
they also include other types of groups, such as the soluble
Baumslag-Solitar groups $\langle x,y \mid y^{-1}xy=x^k \rangle$ with $|k|>1$.

If $\left\langle X \mid R\right\rangle$ is a group presentation, denote its
abelianisation by $\mathrm{Ab}\left\langle X \mid R\right\rangle$.
We call such a presentation an \emph{abelian presentation}, and  
we often use additive notation for its relators.

\begin{lemma}\label{J}
Let $G = G(\c) = \langle \alpha,\beta \rangle$ be a Gc-group,
and let $J = \left\langle \beta^{\alpha^i} \mid i\in \Z \right\rangle$.  
Then $G=J\rtimes \left\langle \alpha \right\rangle$,
\[ J \cong \J := \mathrm{Ab}\left\langle \beta_i\; (i\in \Z) \mid \sum_{i=0}^s c_i\beta_{k+i} \; (k\in \Z) \right\rangle, \]
and $G$ is torsion-free.
\end{lemma}
\begin{proof}
We see from the presentation that
$G = J\rtimes \left\langle \alpha \right\rangle$.
The automorphism $\psi$ of $\J$ given by $\psi(\beta_i) = \beta_{i+1}$ 
corresponds to the automorphism of $J$ given by conjugation by $\alpha$.  We thus have
\[ \J\rtimes \left\langle \alpha \right\rangle \cong \left\langle \beta_i\; (i\in \Z), \alpha \mid [\beta_0, \beta_i], \beta_i^{\alpha} = \beta_{i+1}\; (i\in \Z), \; \beta_0^{c_0}\cdots\beta_{s}^{c_{s}}\right\rangle, \]
since all the other relators of $\J$ are consequences of those listed.
By eliminating the generators $\beta_i$ for $i\neq 0$ using
$\beta_i = \beta_0^{\alpha^i}$
and replacing $\beta_0$ by $\beta$, we arrive at a presentation for $G$,
and so $G \cong \J \rtimes \left\langle \alpha \right\rangle$, with $J\cong \J$.


It remains to show that $G$ is torsion-free, for which it suffices
to show that $J$ is torsion-free. For this, it is enough to show that any
subgroup of $J$ of the form
$J_I:= \left\langle \beta_i \mid i \in I \right\rangle$ is torsion-free, where
$I = \{k,k+1,\ldots,k'\}$.

To show this, we first derive a presentation for $J_I$.
Set $r_k = \sum_{i=0}^s c_i\beta_{k+i}$ for $k\in \Z$.
Suppose that $r: = \sum_{i=k}^{k'} \delta_i\beta_i$ with $\delta_i \in \Z$ is a relator of $J_I$.
Since $r$ is also a relator of $J$, it must be equal
in the free abelian group on the $\beta_i$ to an element of the form
$\sum_{i=l}^{l'} \gamma_i r_i = \sum_{i=l}^{l'}\gamma_i \sum_{j=0}^s c_j\beta_{j+i}$, where $\gamma_l, \gamma_{l'}\neq 0$.  
So we must have $k\leq l$ and $l'+s\leq k'$.  
Conversely, any $r_j$ with $k\leq j\leq k'-s$ is a relator of $J_I$.
Hence $ J_I \cong \mathrm{Ab}\left\langle \beta_i\; (i\in I) \mid r_j\; (k\leq j\leq k'-s) \right\rangle.$

Let $m= k'-s-k$.  Since $J_I$ is a finitely generated abelian group, 
we can find its isomorphism type by computing the Smith normal form of its presentation matrix, 
which is a matrix with $m$ rows and $m+s$ columns of the following form, where all unspecified entries are 0:
\[ M(\c,m) := \left( \begin{array}{cccccccc} 
c_0&c_1 &\ldots& c_s  &      &      &       &\\
   &c_0 &c_1  &\ldots &c_s   &      &       &\\
   &    &.    &\ldots &\ldots&.     &       &\\
   &    &     &.      &\ldots&\ldots&.      &\\
   &    &     &       &c_0   &c_1   &\ldots &c_s\\
\end{array}\right). \]

We will show in Lemma \ref{smithc} below that the Smith normal form of this matrix is
$\left( I_m \mid 0_{m,s} \right)$, and so $J_I$ is a free abelian
group, and hence $J_I$, $J$ and $G$ are torsion-free as claimed.
\end{proof}

\begin{lemma}\label{smithc} Let $\c = (c_0,c_1,\ldots,c_s)\in \Z^{s+1}$ with
$c_0, c_s\neq 0$ and $\gcd(c_0,\ldots,c_s) = 1$.  
Let $m \ge 1$, and let $M$ be the $m \times (m+s)$ matrix $M(\c,m)$ as defined
above.
Then the Smith normal form of $M$ is $\left( I_m \mid 0_{m,s} \right)$.
\end{lemma}
\begin{proof}
We show first that the $m \times m$ minors of $M$ are relatively prime.
Label the columns of $M$ by $C_0, C_1, \ldots, C_{s+m-1}$, and  for $i\notin \{0,\ldots,s\}$ set $c_i = 0$.  
Then $C_i = (c_i, c_{i-1}, \ldots, c_{i-m+1})^{\mathrm{T}}$.  Now for $i\in \{0,\ldots,s\}$, let 
$M^{(i)} = \left( C_i, C_{i+1}, \ldots, C_{i+m-1} \right)$,  and let $d_i = \det(M^{(i)})$.

Then $M^{(0)}$ is an upper triangular matrix with $c_0$ in every position on
the diagonal, so $d_0 = c_0^m$.  Similarly $d_s = c_s^m$.
In general, $M^{(i)}$ has $c_i$ in every position on the diagonal, and every entry below the diagonal is either $0$ or $c_j$ for some $j<i$.  
We can therefore write $ d_i = c_i^m + \sum_{j=0}^{i-1} a_{ij}c_j$, with $a_{ij} \in \Z$.

Let $\delta_i = \gcd(d_0,\ldots,d_i)$ and $\gamma_i = \gcd(c_0,\ldots,c_i)$.  
It is straightforward to show by induction on $i$ that
the prime divisors of $\delta_i$ all divide $\gamma_i$. Hence, since
$\gamma_s = 1$, we have $\delta_s = \gcd(d_0,\ldots,d_s) = 1$, and
the $m \times m$ minors of $M$ are relatively prime, as claimed.

Hence, with $\gamma_i(M)$ and $\sigma_i$ defined as in
Proposition~\ref{smithform}, we have  $\gamma_m(M) = 1$.  
Since $\sigma_i = \gamma_i(M)/\gamma_{i-1}(M)$ and  all $\sigma_i\in \Z$, we must have $\gamma_i(M) = 1$ and $\sigma_i = 1$ for all $1\leq i\leq m$,
which proves the result.
\end{proof}

\subsubsection{Conclusion for Case 1}

We are now ready to classify the groups we have been considering in this subsection.

\begin{proposition}\label{2gen tf}  Let $H = \left\langle a,b\right\rangle$ be torsion-free and satisfy Hypothesis $(*)$.  
Then $H$ is isomorphic to either $\Z \wr \Z$ or a Gc-group.
\end{proposition}
\begin{proof}
Suppose $H$ is not isomorphic to $\Z \wr \Z$.  
Then, as we observed earlier, there exist $s \ge 1$ and $\c = (c_0,\ldots,c_s)\in \Z^{s+1}$ with $c_0,c_s\neq 0$ and $\gcd(c_0,\ldots,c_s) = 1$ 
such that $\sum_{i=0}^s c_i b_i = 0$ (or in multiplicative notation $b_0^{c_0}\cdots b_s^{c_s} = 1$). We choose $\c$ with $s$ minimal.

Let $G = G(\c) = \left\langle \alpha, \beta \mid {\cal R}_{\c} \right\rangle$.
Then since the relations of $G$ are satisfied by $a,b \in H$,
there is an epimorphism $\phi: G\rightarrow H$ with $\phi(\alpha) = a$ and $\phi(\beta) = b$.  
We shall show that $\phi$ is an isomorphism.
For convenience we shall write $\beta_i$ for $\beta^{\alpha^i}$ and $b_i$ for $b^{a^i}$.  Thus ${\cal R}_{\c}=
\left\{ [\beta_0, \beta_i]\; (i\in \Z), \; \beta_0^{c_0}\cdots\beta_s^{c_s} \right\}.$

Let $J$ and $K$ be the subgroups of $G$ defined by $J = \left\langle \beta_i \mid i\in \Z\right\rangle$ and 
$K = \left\langle \beta_0, \beta_1, \ldots, \beta_{s-1} \right\rangle$.  
We saw in Lemma~\ref{J} that $J$ (and hence also $K$) is abelian, 
$G = J\rtimes \left\langle \alpha \right\rangle$ and $G$ is torsion-free.  
If $\beta_0, \beta_1, \ldots, \beta_{s-1}$ were not linearly independent over $\Z$, then their images under $\phi$, 
namely $b_0, b_1, \ldots, b_{s-1}$, would not be linearly independent either, contradicting the minimality of $s$.  
Thus $K$ is free abelian of rank $s$.

We now show that $J/K$ is a torsion group.
We claim that for any $j\in \Z$, there exists some non-zero $\lambda_j\in \Z$ such that $\lambda_j \beta_j\in K$.  
This is clearly true for $0 \leq j\leq s-1$, and also for $j=s$ with
$\lambda_s = c_s$, since $c_s\beta_s = -\sum_{i=0}^{s-1} c_i\beta_i$.
We can then prove the claim by induction for all $j \ge s$ by repeatedly
conjugating by $\alpha$, and for all $j<0$ by conjugating by $\alpha^{-1}$.
So the generators of $J$ are torsion elements modulo $K$, and hence $J/K$ is a torsion group.
So, for each $g\in J$ there exists $t_g\in \N$ such that $t_g g\in K$.    

Let $g\in \ker \phi$.  Since $G = J \rtimes \left\langle \alpha \right\rangle$, 
we can write $g = g'\alpha^n$, where $g'\in J$, $n\in \Z$.  
So $1 = \phi(g) = \phi(g')a^n$, which implies $\phi(g') = a^n = 1$, so $n=0$ and $g\in J$.
Note that $\phi$ is injective on $K$, since $\phi(\beta_0),\ldots,\phi(\beta_{s-1})$ are linearly independent.  
So $\phi(g) = 0$ implies $\phi(t_g g) = 0$ with $t_g g\in K$,
and hence $t_g g=0$.
Since $G$ is torsion-free, this implies $g = 0$.
Thus $\ker \phi$ is trivial and $H$ is isomorphic to $G$.  
\end{proof}

\subsubsection{More about Gc-groups}

We finish this section with an embedding result for Gc-groups and some useful lemmas that follow from it.

\begin{proposition}\label{embed}
Let $G= G(\c)$ be a Gc-group.
Let $\{x_1,\ldots,x_s\}$ be a basis for $\Q^s$ over $\Q$ (the rationals under addition), and let $\Z = \left\langle y \right\rangle$.  
Let $Q = \Q^s\rtimes \Z$,  with the action of $y$ on $\Q^s$ define by
$x_i^y = x_{i+1}$ $(1 \le i < s)$, $x_s^y = \sum_{i=1}^{s}
-(c_{i-1}/c_s)x_i$.
%
%
Then $G$ can be embedded in $Q$.
\end{proposition}
\begin{proof}
Let $G = \langle \alpha, \beta\rangle$.
It is routine to show that $y,x_1 \in Q$ satisfy the defining relations of $G$,
so there is a homomorphism $\theta: G \rightarrow Q$ with
$\theta(\alpha) = y$, $\theta(\beta) = x_1$.
We can show that $\theta$ is injective using a similar argument to the
proof that $\phi$ is injective in Proposition~\ref{2gen tf}.
\end{proof}

\begin{lemma}\label{GBSsub}
Let $G=G(\c)=\left\langle \alpha, \beta\right\rangle$ be a Gc-group.
Then, for any $t \in \N$,  $H = \langle \alpha,\beta^t \rangle$ is a subgroup
of finite index in $G$ with $H \cong G$.
\end{lemma}
\begin{proof}
By Proposition~\ref{embed}, we can identify $G$ with a subgroup of $\Q^s \rtimes \Z$ for some $s$, with $\beta \in \Q^s$. 
Multiplication by $t$ in $\Q^s$ is a $\Z$-module
isomorphism of $\Q^s$, and induces an isomorphism $G \rightarrow H$.
Let $N = G \cap \Q^s$ and $M = H \cap \Q^s$. Then ${|G:H|} = {|N:M|}$.
Now a finitely generated subgroup of $\Q^s$, and hence of $N$,
can be generated by at most $s$ elements.
So $N/M$ has the property that all of its finitely generated subgroups are
finite of order at most $t^s$.  This implies that
$N/M$ is finite of order at most $t^s$, so ${|G:H|}$ is finite.
\end{proof}

\begin{lemma}\label{GBSquot}
Let $G$ be a finitely generated group with finite normal subgroup $T$ 
such that $G/T$ is isomorphic either to the Gc-group $G(\c)$ or to $\Z \wr \Z$.
Then $G$ has a finite index subgroup isomorphic to $G(\c)$ or $\Z \wr \Z$ respectively.
\end{lemma}
\begin{proof}
Suppose that $G/T$ is isomorphic to the Gc-group
$G(\c)=\langle \alpha T,\beta T \rangle$. By the previous lemma, for any $t\in \N$,  
$\langle \alpha T,\beta^t T \rangle$ is isomorphic to $G(\c)$ and has finite index in
$G/T$, so by replacing $\beta$ by a suitable power we may assume that
$\beta \in C_G(T)$.  Define $\beta_i = \beta^{\alpha^i}$ for $i \in \Z$,
and denote the exponent of $T$ by $m$.
Then the elements $\beta_i^m$ all commute so
by passing to a finite index subgroup again, we may assume that
the $\beta_i$ all commute. Then
\[ (\beta_0^m)^{c_0}(\beta_1^m)^{c_1} \cdots (\beta_s^m)^{c_s} =
(\beta_0^{c_0}\beta_1^{c_1} \cdots \beta_s^{c_s})^m = 1, \]
where $\c = (c_0,\ldots,c_s)$, so by passing to a finite index subgroup a third time 
we may assume $\langle \alpha,\beta \rangle \cong G(\c)$ is a complement of $T$ in $G$ and hence has finite index in $G$.
The proof for $G/T \cong \Z \wr \Z$ is left to the reader.
\end{proof}

Henceforth we shall only be interested in Gc-groups that are not virtually
abelian.  We shall call a non virtually abelian Gc-group a {\em proper Gc-group}.

\subsection{Case 2:  $b$ is a torsion element}

\begin{proposition}\label{2gen tor}  Let $H = \left\langle a,b\right\rangle$ with $b$ a torsion element and assume $H$ satisfies Hypothesis $(*)$.  
Then $H$ has a subgroup isomorphic to $C_p \wr \Z$ for some prime $p$.
\end{proposition}
\begin{proof}

The proof is by induction on the order $|b|$ of $b$.
If $|b| = p$ is prime, then $H\cong C_p\wr \Z$.  For if not, then some relation $\sum_{i=r}^s \gamma_i b_i = 0$ with $p \nd \gamma_r, \gamma_s$
is true in $H$, and by conjugating by powers of $a$, we see that
$B = \langle b_r, \ldots, b_{s-1}\rangle$ and $H$ is virtually abelian.

Suppose $|b|=np$ with $p$ prime and $n>1$, and let $H_1 = \left\langle a, b^n\right\rangle$.  
Since $|b^n|=p$, either $H_1\cong C_p\wr \Z$ or $H_1$ is virtually abelian.  
In the first case we are done, so suppose that $H_1$ is virtually abelian.
Then $\left\langle b_i^n \mid i\in \Z\right\rangle$ is finitely generated 
and so some power of $a$, say $a^k$, centralises $b^n$.  
If $\left\langle a^k,b\right\rangle$ is virtually abelian, then $\left\langle b_{ik} \mid i\in \Z\right\rangle$ is finitely generated, 
hence finite.  But then some power of $a^k$ centralises $b$ and hence $B$, implying that $H$ is virtually abelian.
Thus $\left\langle a^k, b\right\rangle$ is not virtually abelian,
so we can replace $a$ by $a^k$ and thereby
assume that $a$ itself centralises $b^n$, and hence
$b_i^n = b_j^n$ for all $i,j\in \Z$.  
Put $b' = b_0b_1^{-1}$, $H_2 = \left\langle a, b'\right\rangle$,  
$b'_i = (b')^{a^i}$ and
$\widehat{B} = \left\langle b'_i \mid i\in \Z\right\rangle$. 
Then ${|B:\widehat{B}|}$ and hence ${|H:H_2|}$ is finite, so $H_2$ cannot
be virtually abelian.  Also, $(b')^n = b_0^nb_1^{-n} = 1$, so
$|b'|$ divides $n$, and
the result follows by the inductive hypothesis applied to $H_2$.  
\end{proof}

\subsection{The general case}


\begin{proposition}\label{2gen 3}  Let $H = \left\langle a, b\right\rangle$ and assume $H$ satisfies Hypothesis $(*)$.  
Then one of the following holds:
\begin{enumerate}
\item $H$ has a subgroup isomorphic to $C_p\wr \Z$ for some prime $p$;
\item $H$ has a subgroup that is isomorphic to
$\Z \wr \Z$ or to a proper Gc-group.
\end{enumerate}
\end{proposition}
\begin{proof}
If $|b|$ is finite, then we are done by Proposition~\ref{2gen tor}, so assume
$|b|$ is infinite.  Any torsion elements of $H$ are contained in $B$.
Since $B$ is abelian, its torsion elements form a subgroup $T$, which is normal in $G$.  

Since $B$ is finitely generated (by $b$) as a $\Z[a]$-module,
Hilbert's Basis Theorem tells us that all $\Z[a]$-submodules of $B$ are finitely generated.  
In particular, $T$ is finitely generated as a $\Z[a]$ module.  
So there is a finite subset $\{t_1,\ldots,t_n\}$ of $T$ such that
\[ T = \left\langle t_i^{p(a)} \mid p(a)\in \Z[a], 1\leq i\leq n\right\rangle = \left\langle t_i^{a^j} \mid j\in \Z, 1\leq i\leq n\right\rangle.\]

By Proposition~\ref{2gen tor}, we can assume that 
$\left\langle a,t\right\rangle$ is virtually abelian for all $t\in T$,
and hence that $T$ is finite.  
So $H/T$ cannot be virtually abelian.
Now $H/T$ is a torsion-free group satisfying Hypothesis $(*)$ and the result
follows from Proposition~\ref{2gen tf} and Lemma~\ref{GBSquot}.
\end{proof}

\section{Metabelian groups}\label{metasec}

We require two further results in order to deal with metabelian groups that have no subgroups satisfying Hypothesis $(*)$.

\begin{proposition}\label{polycyclic}  
Let $G$ be a finitely generated group with a normal abelian subgroup $N$ such that $G/N$ is free abelian.  If for all $a\in G$, $b\in N$, the subgroup $\langle a,b\rangle$ is virtually abelian, then $G$ is polycyclic.
\end{proposition}
\begin{proof}
Suppose that $\langle a,b\rangle$ virtually abelian for all $a\in G$, $b\in N$.
Then $\left\langle b^{a^i} \mid i\in \Z \right\rangle$ is finitely generated for any $a\in G$, $b\in N$.  
We will show that $N$ is finitely generated and thus $G$ is polycyclic.

If $G/N = \langle Na_1, \ldots, Na_k \rangle$, then  any $g\in G$ can be
written in the form $g = b a_1^{r_1}\ldots a_k^{r_k}$, where
$b\in N, r_i\in \Z$.  It can be shown by a straightforward induction argument
that, for $b\in N$, $H_b := \left\langle b^g \mid g\in G \right\rangle$
is finitely generated.
Since $G/N$ is finitely presented, by Proposition~\ref{G/N}, $N$ is the normal
closure in $G$ of a finite set of elements and hence $N$ is finitely generated
as claimed.  
\end{proof}

\begin{lemma}\label{polylemma}
If a group $G$ is polycyclic and not virtually abelian, then $G$ has a subgroup
that is isomorphic to a proper Gc-group.
\end{lemma}
\begin{proof}
By the second and third paragraphs of the proof of Theorem 16 in \cite{HRRT},
there exist $a\in G$ and a non-trivial free abelian subgroup $N\lhd G$ such that $\langle a, N\rangle$ is not virtually abelian.
Let $N = \langle b_1,\ldots,b_k \rangle$. If $\langle a,b_i \rangle$ is
virtually abelian for some $i$, then there exist $s_i,t_i \in \N$ such that 
$a^{s_i} \in C_G(b_i^{t_i})$.
If this were true for all $i$ with $1 \le i \le k$, then, putting $s= \max\{s_i\}$ and $t= \max\{t_i\}$, $\langle a^s, N^t \rangle$
would be an abelian subgroup of finite index in $\langle a, N\rangle$,
contrary to assumption. So at least one of the subgroups $\langle a,b_i \rangle$
is isomorphic to a proper Gc-group, 
by Proposition~\ref{2gen tor}.
\end{proof}
   
We are now ready to prove the full result for metabelian groups.

\begin{theorem}\label{metabelian}  Let $G$ be a finitely generated metabelian group that is not virtually abelian.  
Then $G$ has a finitely generated subgroup isomorphic to one of the following:
\begin{enumerate}
\item $\Z \wr \Z$; 
\item $C_p \wr \Z$ for some prime $p$;
\item A proper Gc-group.
\end{enumerate}
\end{theorem}
\begin{proof}
By replacing $G$ by a finite index subgroup if necessary, we may assume that 
$G$ has a normal abelian subgroup $N$ such that $G/N$ is free abelian.
The result follows from Proposition \ref{polycyclic} and Lemma~\ref{polylemma}
if $\left\langle a, b \right\rangle$ is virtually abelian 
for all $a\in G$, $b\in N$, and from Proposition~\ref{2gen 3} otherwise.
\end{proof}

\section{Torsion-free soluble groups}\label{tfsolsec}

\begin{proposition}\label{schurcol} Let $G$ be a torsion-free group and $A$ an abelian subgroup of finite index in $G$.  
Then $C_G(A)$ is abelian.
\end{proposition}
\begin{proof}  Since $A$ is abelian, $A\leq Z(C_G(A))$, and  
hence $Z(C_G(A))$ has finite index in $C_G(A)$. It follows from Proposition~\ref{schur} that
$C_G(A)'$ is finite and thus in fact trivial, since $G$ is torsion-free.
\end{proof}

\begin{proposition}\label{countgen}  
Let $G$ be a countable torsion-free locally virtually abelian group  
and suppose that $G$ does not contain free abelian subgroups of arbitrarily high finite rank.
Then $G$ has a finite index normal abelian subgroup that is characteristic.
\end{proposition}
\begin{proof}
Let $k\in \N$ be maximal such that $G$ has a free abelian subgroup of rank $k$.  
We can choose elements $g_i \in G$ for $i \in \N$ such that
 $G = \langle g_i \mid i\in \N\rangle$ and $\langle g_1,\ldots,g_k\rangle
\cong \Z^k$.

For $i\geq k$, let $H_i = \langle g_1,\ldots,g_i\rangle$ and let $A_i$ be a
subgroup of $H_i$ that is maximal subject to having finite index in $H_i$ and
being isomorphic to $\Z^k$.  
By Proposition \ref{schurcol}, $C_{H_i}(A_i)$ is abelian, and so $C_{H_i}(A_i) = A_i$.
 
Suppose $B_i$ is another rank $k$ abelian subgroup of $H_i$.  
Then $A_i\cap B_i$ has finite index in $B_i$, so has rank $k$, so has finite index in $A_i$ and hence in $H_i$.  
Proposition \ref{schurcol} implies that $C_{H_i}(A_i\cap B_i)$ is abelian.  
But $A_i,B_i\leq C_{H_i}(A_i \cap B_i)$, and so by the maximality of $A_i$, we have $C_{H_i}(A_i\cap B_i) = A_i$, so $B_i\leq A_i$.  
This proves that $A_i$ is the \emph{unique} maximal abelian subgroup of rank $k$ in $H_i$.  
It follows that $A_i\unlhd H_i$ and $A_i\leq A_{i+1}$ for all $i\geq k$.
Since $A_i$ is self-centralising, $H_i/A_i$ is isomorphic to a subgroup of $\Aut(A_i) \cong \GL(k,\Z)$.

For $i\geq k$, let $\iota_i$ be the natural embedding of $H_i$ in $H_{i+1}$, $\phi_i$ the natural homomorphism from $H_i$ to $H_i/A_i$ 
and let $\psi_i:H_i/A_i \rightarrow H_{i+1}/A_{i+1}$ be given by $xA_i \mapsto xA_{i+1}$.
Then $\psi_i\circ \phi_i =  \phi_{i+1}\circ\iota_i$
and so $\ker \psi_i\circ \phi_i = \ker \phi_{i+1}\circ\iota_i$.  Now
\[\ker \phi_{i+1}\circ\iota_i = H_i\cap \ker \phi_{i+1} = H_i\cap A_{i+1} = A_i,\]
where the last equality follows from the maximality of $A_i$.
Thus $\ker \psi_i\circ \phi_i = A_i$.  Since already $\ker \phi_i = A_i$, this implies that $\psi_i$ is injective.

Hence, since each $H_i/A_i$ is isomorphic to a finite subgroup of $\GL(m,\Z)$, Proposition \ref{GLbound} 
implies that there exists $m\in \N$ such that $H_i/A_i$ is isomorphic to $H_m/A_m$ for all $i\geq m$.

We have a commutative diagram:
\[\begin{CD}
H_m          @>\iota_m>>   H_{m+1}           @>\iota_{m+1}>> \ldots  @>\iota_{i-1}>>   H_i        @>\iota_i>>  H_{i+1}\\
@VV\phi_m V\               @VV\phi_{m+1} V                   @VV V                                @VV\phi_i V @VV\phi_{i+1} V\\
H_m/A_m      @>\psi_m>>    H_{m+1}/A_{m+1}   @>\psi_{m+1}>>  \ldots  @>\psi_{i-1}>>    H_i/A_i    @>\psi_i>> H_{i+1}/A_{i+1}\\
\end{CD}\]

For $i\geq m$, define $\theta_i: H_i\rightarrow H_m/A_m$ by $\theta_m = \phi_m$ and for $i>m$ 
\[\theta_i = \psi_m^{-1}\circ\psi_{m+1}^{-1}\ldots\circ\psi_{i-1}^{-1}\circ \phi_i.\]
Since each $\phi_i$ is an epimorphism and all the  $\psi_j^{-1}$ ($j\geq m$) are isomorphisms, 
each $\theta_i$ is an epimorphism. Moreover,  we have $\theta_{i+1} \circ 
\iota_i = \theta_i$; that is, $\theta_{i+1}(x) = \theta_i(x)$ for $x \in H_i$.

Define $\theta: G\rightarrow H_m/A_m$ by $\theta(x) = \theta_i(x)$ where $i = \min\{j\geq m \mid x\in H_j\}$.  
Then $\theta$ is a well defined epimorphism with kernel $A := \bigcup_{i\in \N} A_i$.
So $A$ is abelian of finite index in $G$.

Finally, we show that $A$ is characteristic in $G$.
Let $B$ be an abelian subgroup of $G$ containing $\Z^k$, and choose $h_i$
such that $B = \left\langle h_i \mid i\in \N \right\rangle$. 
Put $B_i =\langle h_1,\ldots,h_i\rangle$ for $i \ge k$ and assume that $B_k \cong \Z^k$. 
Then each $B_i$ is contained in some $H_{j_i}$, and hence in $A_{j_i}$.  So
$ B = \bigcup_{i\in \N} B_i \leq \bigcup_{i\in \N} A_{j_i} \leq \bigcup_{i\in \N} A_i = A.  $ 
Thus $A$ is the unique maximal abelian subgroup of $G$ containing $\Z^k$ and so is characteristic in $G$.
\end{proof}

\begin{lemma}\label{Zinf-tf}
Let $G$ be a countable torsion-free locally virtually abelian group and
let $N$ be an abelian normal subgroup of $G$ such that $G/N\cong \Z^\infty$.
Then $G$ has a subgroup isomorphic to $\Z^\infty$.
\end{lemma}

(We recall that $\Z^\infty$ denotes a free abelian group of countably infinite rank.)

\begin{proof}  
If $N$ contains $\Z^k$ for all $k\in \N$ then we are done, since $N$ is abelian, so
assume $t$ is maximal with $\Z^t\leq N$.
Write $G = \bigcup_{i\in \N} G_i$, with each $G_i$ finitely generated (and hence virtually abelian),
$G_i < G_{i+1}$ for all $i\in \N$, and each $N_i := G_i \cap N$ isomorphic to $\Z^t$. 

Let $A_i = C_{G_i}(N_i)$ and let $B_i$ be an abelian subgroup of finite index in $G_i$.
Then $B_i$ centralises the finite index subgroup $B_i \cap N_i$ of $N_i$.
It is easily seen that an automorphism of a free abelian group of finite rank
that centralises a subgroup of finite index must be trivial, so
$B_i$ must centralise $N_i$. Thus $B_i\leq A_i$, so ${|G_i:A_i|}$ is finite.
Each $N_i$ has finite index in $N_{i+1}$, so 
for all $i$ we have $A_i\leq C_{G_{i+1}}(N_{i+1}) = A_{i+1}$.

Also $A_i$ must be abelian, since otherwise there exist $x,y\in A_i$ with
$[x,y] = z\in N_i\setminus \{1\}$, but then $\langle x,y\rangle$ is not virtually abelian, contrary to hypothesis.
Thus $\bigcup_{i\in \N} A_i$ is abelian and contains finitely generated
subgroups of arbitrarily high rank, so it must contain $\Z^\infty$.
\end{proof}

\begin{lemma}\label{Z^inf}
Let $G$ be a locally virtually abelian group with a normal torsion subgroup
$T$ such that $G/T\cong \Z^\infty$.  
Then $G$ has a subgroup isomorphic to $\Z^\infty$.
\end{lemma}

\begin{proof}
Let  $g_1,g_2,\ldots$ be an irredundant generating set for $G$ modulo $T$.  
Then for all $i\in \N$, $G_i := \langle g_1,\ldots,g_i\rangle$ is virtually abelian, 
so $T_i := G_i\cap T$ is finite.  Also $G_i/T_i\cong \Z^i$.

We shall construct a chain of subgroups $H_1 < H_2 < \cdots < H_i <  \cdots$
such that $H_i \le G_i$ with ${|G_i:H_i|}$ finite  and $H_i \cong \Z^i$.
Suppose that we have defined $H_1, \ldots,H_i$ for some $i \ge 0$.
Then $g_{i+1} \in G_{i+1}$ centralises $H_iT_{i+1}/T_{i+1}$, 
while some power $g_{i+1}^k$ of $g_{i+1}$ centralises $T_{i+1}$.  
If $l = k|T_{i+1}|$, then $g_{i+1}^l$ centralises $H_i$.
Thus $H_{i+1} := \langle H_i, g_{i+1}^l\rangle$ 
has finite index in $G_{i+1}$ and is isomorphic to $\Z^{i+1}$.
So we can construct the chain as claimed, and
$\bigcup_{i\in \N} H_i \cong \Z^\infty$.
\end{proof}

\begin{proposition}\label{Zinf-sol}
Let $G$ be a countable soluble locally virtually abelian group, and suppose 
that $G$ has subgroups isomorphic to $\Z^k$ for all $k\in \N$.
Then $G$ has a subgroup isomorphic to $\Z^\infty$.
\end{proposition}
\begin{proof}
The proof is by induction on the derived length of $G$.
The statement is true for abelian groups.  Now suppose it is true for groups of
derived length at most $n$, and let $G$ be locally virtually abelian
of derived length $n+1$, with $\Z^k\leq G$ for all $k\in \N$.

Let $N = G^{(n)}$.  If $\Z^k\leq N$ for all $k\in \N$ then we are done, since $N$ is abelian.
So suppose there exists $t\in \N$ maximal such that $\Z^t\in N$.
This means that $G/N$ must contain $\Z^k$ for all $k\in \N$.
By the induction hypothesis, $G/N$ has a subgroup isomorphic to $\Z^\infty$.
We can assume without loss of generality that $G/N\cong \Z^\infty$.

Let $T$ be the torsion subgroup of $N$.  Then $N/T$ is torsion-free, and so $G/T$ and $N/T$
satisfy the hypothesis of Lemma~\ref{Zinf-tf}.  Thus we can replace $G$ by a subgroup and
assume $G/T\cong \Z^\infty$.  Then by Lemma~\ref{Z^inf}, $G$ has a subgroup isomorphic to
$\Z^\infty$.
\end{proof}

We can now prove our main result for torsion-free finitely generated soluble
groups.

\begin{theorem}\label{tfsol}  Let $G$ be a finitely generated torsion-free soluble group.
Then at least one of the following holds:
\begin{enumerate}
\item $G$ is virtually abelian;
\item $G'$ has a subgroup isomorphic to $\Z^\infty$;
\item $G$ has a subgroup isomorphic to a proper Gc-group.
\end{enumerate}
\end{theorem}
\begin{proof}

Let $n$ be the derived length of $G$.  The proof is by induction on $n$.
The statement is vacuously true for $n=1$, and true for $n=2$ ($G$ metabelian) 
by Theorem \ref{metabelian}. So assume that $n \ge 2$ and that the result is true for
groups with derived length less than $n$. By applying the result inductively
to $G'$ we may assume $G'$ is locally virtually abelian.

We may assume that $G'$ does not contain $\Z^k$ for all $k\in \N$, since otherwise we are done by Proposition~\ref{Zinf-sol}.
Since $G'$ is countable, Proposition \ref{countgen} shows that $G'$ 
has a finite index normal abelian subgroup $A$ which is characteristic in $G'$ and hence $A \lhd G$.  
Now $G/A$ is finitely generated with $(G/A)/(G'/A)\cong G/G'$ abelian, while $G'/A$ is finite.  So $G/A$ is virtually abelian.
But since $A$ is an abelian normal subgroup of $G$, this implies that $G$ itself is virtually metabelian, and the result follows
from Theorem~\ref{metabelian}.
\end{proof}

\section{The main result}\label{mainsec}

We require just one further lemma before proceeding to prove our result for finitely generated soluble groups in general.

\begin{lemma}\label{(G/N)'}
Let $G$ be a group with normal subgroup $N$ and let $H$ be a subgroup of $(G/N)'$.  
Then $G'$ has a subgroup $K$ such that ${K/(N\cap K)}\cong H$.
\end{lemma}
\begin{proof}
This follows from $(G/N)' = G'N/N \cong G'/(G' \cap N)$.
\end{proof}

\begin{theorem}\label{main}
Let $G$ be a finitely generated soluble group.
Then at least one of the following holds:
\begin{enumerate}
\item $G$ is virtually abelian;
\item $G'$ has a subgroup isomorphic to $\Z^\infty$;
\item $G$ has a subgroup isomorphic to a proper Gc-group; 
\item $G$ has a finitely generated subgroup $H$ with an infinite
normal torsion subgroup $U$, such that $H/U$ is
either free abelian or a proper Gc-group.
\end{enumerate}
\end{theorem}
\begin{proof}
The proof is by induction on the derived length of $G$. We have already proved
it for groups of derived length at most two and when $G$ is torsion-free
(Theorems \ref{metabelian} and \ref{tfsol}).
Suppose the result holds for groups of derived length at most $n \ge 2$, and
let $G$ have derived length $n+1$. So the result holds for $G/N$,
where $N=G^{(n)}$, and we may assume that $G'$ is locally virtually abelian. Let $T$ be the torsion subgroup of $N$.
For the remainder of the proof, we shall only use the fact that $N$ is abelian,
and not that $N=G^{(n)}$. This enables us to replace $G$ by a subgroup when
convenient.

If (i) holds for $G/N$, then $G$ is virtually of derived length $2$, and the
result follows from Theorem~\ref{metabelian}.  

Suppose that (ii) holds for $G/N$; that is, $(G/N)'$ contains a subgroup
isomorphic to $\Z^\infty$.  
Applying Lemma \ref{(G/N)'}, let $K$ be a subgroup of $G'$ such that $K/(N\cap K)\cong \Z^\infty$.
Then $K/(T\cap K)$ is torsion-free and so, by Lemma \ref{Zinf-tf}, there is also a $\Z^\infty$ subgroup in $K/(T\cap K)$.
By Lemma \ref{Z^inf}, this implies that $K$ (and hence $G'$) has a subgroup isomorphic to $\Z^\infty$
and so (ii) holds for $G$.

Suppose that (iii) holds for $G/N$.
Then, by replacing $G$ by a subgroup, we may assume that $G/N$ is isomorphic
to a proper Gc-group.
So $G/T$ is torsion-free and not virtually abelian,
and hence by Theorem~\ref{tfsol} one of (ii), (iii) holds for $G/T$.
If (ii) holds for $G/T$, then (ii) holds for $G$ by Lemma \ref{Z^inf}.
So suppose that (iii) holds for $G/T$. By passing to a subgroup we may
assume that $G/T$ is a proper Gc-group.
If $T$ is finite, then the result follows from Lemma~\ref{GBSquot},
whereas if $T$ is infinite, then $G$ satisfies (iv).

Finally suppose that $G/N$ satisfies (iv). We may assume that $G/N$ has an
infinite normal torsion subgroup $U/N$, where $G/U$ is either torsion-free
abelian or a proper Gc-group.
If $N=T$ then $G$ satisfies (iv), so suppose not.
Since $N$ is abelian, if $N/T$ contains subgroups isomorphic to $\Z^k$
for all $k \in \N$, then $N$ contains $\Z^\infty$, so we may assume that there is a
largest $k>0$ such that $\Z^k \le N/T$.

Suppose that $U$ is generated by elements $g_i$ ($i \in \N$) and,
for $i \in \N$, define $U_i = \langle g_1,\ldots,g_i \rangle$ and
$N_i = U_i \cap N$.
A finitely generated soluble torsion group is finite, so the groups
$U_i/N_i \cong U_iN/N$ are all finite with $\bigcup_{i\in \N} U_iN/N=U/N$.
The groups $N_iT/T$ are all finitely generated free abelian groups
with $\bigcup_{i\in \N} N_iT/T=N/T$ and so for all sufficiently
large $i$ we have $N_iT/T \cong \Z^k$.
For $i \in \N$, define the subgroup $C_i$ by $U_i \cap T \le C_i \le U_i$ and $C_iT/T = C_{U_iT/T}(N_iT/T)$.
By Proposition \ref{GLbound}, there is a constant $L$ such that ${|U_i:C_i|} \le L$
for all $i$. Since $N_i \le N_{i+1}$ for each $i$, we have
$C_{i+1} \cap U_i \le C_i$, and hence ${|U_i:C_i|} \le {|U_{i+1}:C_{i+1}|}$.
So, for sufficiently large $i$, we have $N_iT/T \cong \Z^k$ and
${|U_i:C_i|} = L$ for some constant $L$, and so we must also have
$C_{i+1} \cap U_i = C_i$. So $C := \bigcup_{i\in \N} C_i$
satisfies $T < C$ and $C/T = C_{U/T}(N/T)$, and ${|U:C|}=L$.

Now $C \lhd G$. If $G/U$ is torsion-free abelian, then $G/C$ is virtually
abelian, whereas if $G/U$ is a proper Gc-group, then by Lemma~\ref{GBSquot}
$G/C$ is virtually a proper Gc-group.
In either case, we can replace $G$ by a finite index subgroup
containing $C$ such that $G \cap U = C$, and thereby assume that $U=C$
and hence $U_i=C_i$ for all $i$.
So $N_iT/T \le Z(U_iT/T)$ for each $i$, and by Proposition~\ref{schur} $U_i'T/T$ is finite.
Since $U' = \bigcup_{i\in \N} U_i'$, $U'T/T$ and hence also $U'T$
is a torsion group.
Let $V/U'T$ be the torsion subgroup of the abelian group $U/U'T$.
So $U/V$ and hence also $G/V$ is torsion-free, and hence
by Theorem~\ref{tfsol} one of (i), (ii), (iii) holds for $G/V$.

If $G/V$ is virtually abelian then, since $G$ is not virtually
abelian, $V$ must be infinite, and so $G$ satisfies (iv).
If $G/V$ satisfies (ii) then so does $G$ by Lemmas~\ref{(G/N)'} and~\ref{Z^inf}.
If (iii) holds for $G/V$ then, by replacing $G$ by a subgroup, we may assume that $G/V$ is isomorphic to a proper Gc-group.
Then the result follows by Lemma~\ref{GBSquot}
if $V$ is finite, and $G$ satisfies (iv) if $V$ is infinite.
\end{proof}

\section*{Acknowledgement}

The first author was supported by a Vice-Chancellor's Scholarship from the University of Warwick.

\end{document}